\documentclass[12pt]{article}

\usepackage{amssymb,amsmath,amsthm,amsfonts,latexsym,amscd}
\usepackage{fullpage}
\usepackage{float}

\def\modd#1 #2{#1\ ({\rm mod}\ #2)}

\theoremstyle{plain}
\newtheorem{theorem}{Theorem}
\newtheorem{corollary}[theorem]{Corollary}
\newtheorem{lemma}[theorem]{Lemma}

\newtheorem{observation}[theorem]{Observation}

\theoremstyle{definition}

\theoremstyle{remark}

\DeclareMathOperator{\sgn}{sgn}
\def\ng{N{\o}rg{\aa}rd}
\def\Enn{\mathbb{N}}
\def\Zee{\mathbb{Z}}

\author{Christopher Drexler-Lemire and Jeffrey Shallit\footnote{Corresponding author.} \\
School of Computer Science \\
University of Waterloo \\
Waterloo, ON  N2L 3G1 \\
Canada \\
{\tt drexlerlemirec@gmail.com} \\
{\tt shallit@cs.uwaterloo.ca} }

\title{Notes and Note-Pairs in \ng's Infinity Series}

\begin{document}

\maketitle

\begin{abstract}
The Danish composer Per \ng\ defined the ``infinity series''
${\bf s} = (s(n))_{n \geq 0}$ by the rules
$s(0) = 0$,
$s(2n) = -s(n)$ for $n \geq 1$, and $s(2n+1) = s(n)+1$ for $n \geq 0$;
it figures
prominently in many of his compositions.  Here we give several new
results about this sequence:  first, the set of
binary representations of the positions
of each note forms a context-free language that is not
regular; second, a complete characterization of exactly which
note-pairs appear; third, that consecutive occurrences of identical
phrases are widely separated.  We also consider to what extent the
infinity series is unique.
\end{abstract}

\section{Introduction}

The Danish composer Per \ng\  constructed an infinite sequence of
integers, ${\bf s} = (s(n))_{n \geq 0}$,
called by him the
{\it Uendelighedsr{\ae}kken} or ``infinity series'',\footnote{Mathematicians
would call the ``infinity series'' a sequence, not a series, but we have chosen to retain the original terminology when referring to \ng's sequence.} 
using the rules
$$ s(n) = \begin{cases}
	0, & \text{if $n = 0$;} \\
	-s(n/2), & \text{if $n$ even;} \\
	s({{n-1} \over 2}) + 1, & \text{if $n$ odd.}
\end{cases}$$	
Starting at some base note, such as $G = 0$, this sequence specifies
the number of half-steps away from the base note, with positive numbers
representing notes of higher pitch and negative numbers representing
notes of lower pitch.  
It figures prominently in many of his compositions, such as 
{\it Voyage into the Golden Screen} \cite{Norgard:2010} and
{\it Symphony No.~2}.

The following table gives the first few terms of the sequence.

\begin{table}[H]
\begin{center}
\begin{tabular}{c|cccccccccccccccccccc}
$n$ & 0 & 1 & 2 & 3 & 4 & 5 & 6 & 7 & 8 & 9 & 10 & 11 & 12 & 13 & 14 
& 15 & 16 & 17 & 18 & 19 \\
\hline
$s(n)$ & 0 & 1 & $-1$ & 2 & 1 & 0 & $-2$ & 3 & $-1$ & 2 & 0 & 1
& 2 & $-1$ & $-3$ & 4 & 1 & 0 & $-2$ & 3 \\
\end{tabular}
\end{center}
\end{table}

Although the infinity series has received some study
\cite{Kullberg:1996,Mortensen:2014a}, it has a rich mathematical
structure that has received little attention.
In this article we examine some novel aspects of the sequence.

We fix some notation used throughout the paper.  By $(n)_2$ we mean the
binary string, having no leading zeros,
representing $n$ in base $2$.  Thus, for example, $(43)_2 = 101011$.
Note that $(0)_2$ is the
empty string $\epsilon$.  If $w$ is a binary string, possibly
with leading zeros, then by $[w]_2$ we
mean the integer represented by $w$.  Thus, for example, $[0101]_2 = 5$.

By a {\it block} we mean a finite list of consecutive terms of the
sequence.  When we interpret the sequence musically, we call this
a {\it phrase}.  The block of length $j$ beginning at position $i$
of the infinity series is denoted by ${\bf s}[i..i+j-1]$.  If $x$
is a block, then by $|x|$ we mean the length of, or number of notes in,
the block $x$.

We recall two basic facts about the infinity series, both of
which follow immediately from the defining recurrence.

\begin{observation} If
a number $a$ occurs at an even position $n = 2k$, then
$1-a$ occurs at position $n = 2k+1$. If a number $b$ occurs
at an odd position $n = 2k+1$, then $1-b$ occurs at position $n = 2k$.
\label{obs1}
\end{observation}

\begin{observation}
The infinity series is the fixed point of the
map $g$ that sends each integer $a$ to the pair $(-a, a+1)$.  
\label{obs2}
\end{observation}

\section{Evaluating the infinity series}

It is useful to have a formula to compute $s(n)$ directly from $(n)_2$,
the base-$2$ expansion of $n$.    The result below can be compared
with an essentially equivalent formulation by Mortensen \cite{Mortensen:2014b}.

\begin{lemma}
Let $a_1, a_2, \ldots, a_n, b_1, b_2, \ldots, b_n$ be
integers with $a_1, a_, \ldots, a_j \geq 0$ and
$b_1, b_2, \ldots, b_j \geq 1$.  
If $w = 1^{b_1} 0^{a_1} \cdots 1^{b_n} 0^{a_n}$
then
$$ s([w]_2) = \sum_{1 \leq j \leq n} (-1)^{a_j + \cdots + a_n} b_j .$$
\label{lem1}
\end{lemma}

\begin{proof}
By induction on $n$.  The base case is $n = 1$.
In this case $ w = 1^{b_1} 0^{a_1}$, so
$[w]_2 = 2^{b_1 - 1} 2^{a_1}$.
Then $s([w]_2) = s( (2^{b_1} - 1) 2^{a_1}) = (-1)^{a_1} s(2^{b_1} - 1)
	= (-1)^{a_1} b_1 ,$
as desired.

For the induction step, assume the result is true for $n$.
Consider $w' = 1^{b_1} 0^{a_1} \cdots 1^{b_n} 0^{a_n} 1^{b_{n+1}} 0^{a_{n+1}} $.
Then, applying the rules of the recursion, we have
$s([w']_2) = (-1)^{a_{n+1}} s([w''_2])$, where
$w'' = 1^{b_1} 0^{a_1} \cdots 1^{b_n} 0^{a_n} 1^{b_{n+1}}$.
Again applying the rules, we have
$s([w'']_2) = b_{n+1} + s([w]_2)$.  
Putting this all together, and using induction, we get
$$s([w']_2) = (-1)^{a_{n+1}} ( b_{n+1} + s([w_2]) )$$
or
$$ s([w']_2) = \sum_{1 \leq j \leq n+1} (-1)^{a_j + \cdots + a_{n+1}} b_j,$$
as desired.
\end{proof}

\section{Note positions}

Given any note $a$, we can consider the set $s^{-1} (a)$ of natural numbers
$n$ such that $s(n) = a$.   For example, for $a = 0$, we have
$$ s^{-1} (0) = \{ 0, 5,
10,17,20,27,34,40,45,54,65,68,75,80,85,90,99,105,108,
\ldots \}.$$
It is then natural to wonder about the complexity of specifying these
note positions.  

The American linguist Noam Chomsky invented a famous hierarchy of
distinctions among formal languages \cite{Chomsky:1956}.  The two lowest
levels of this hierarchy are the regular languages (those accepted
by a finite-state machine) and the context-free languages (those
accepted by a finite-state machine with an auxiliary pushdown stack).
Here we show that the language of binary representations of $s^{-1} (a)$
is context-free, but not regular.

\begin{theorem} 
For each integer $a$, the language $L_a = (s^{-1} (a))_2$ is context-free
but not regular.
\end{theorem}

\begin{proof}
It suffices to explain how $L_a$ can be accepted by a pushdown automaton $M_a$.
We assume the reader is familiar with the basic notation and terminology
as contained, for example, in \cite{Hopcroft&Ullman:1979}.

The first part of the construction is the same for all $a$.

We will design $M_a$ such that, on input $n$ in base
2 (starting from the most significant digit), $M_a$ ends up 
with $|s(n)|$ counters on its stack, with the sign $m := \sgn(s(n))$
stored in the state.
We also assume there is an initial stack symbol $Z$.

To do this, we use the recursion
$s(2n) = -s(n)$ and $s(2n+1) = s(n)+1$.   As we read the bits of $n$,
\begin{itemize}
\item if the next digit read is 0, set $m := -m$;
\item if the next digit read is 1, and $m$ is 0 or $+1$, push a counter on the
stack, and set $m := +1$;
\item if the next digit read is 1, and $m = -1$, pop a counter from the stack
and change $m$ to 0 if $Z$ is now on the top of the stack.
\end{itemize}

The rest of $M_a$ depends on $a$.  From each state where $m = \sgn(a)$,
we allow an $\epsilon$-transition to a state that attempts to pop off
$|a|$ counters from the stack and accepts if and only if this succeeds,
the stored sign is correct,
and $Z$ is on top of the stack.  This completes the sketch of our
construction, and proves that
$L_a$ is context-free.

Next, we prove that $L_a$ is not regular.  Again, we assume the reader is
familiar with the pumping lemma for regular languages, as
described in \cite{Hopcroft&Ullman:1979}.  Consider
$L := L_a \ \cap \ 1^* 0 1^*$.  From Lemma~\ref{lem1}, we know that
if $n = [1^b 0 1^c]_2$, then $s(n) = c-b$.    It follows 
that 
$$L = \lbrace 1^i 0 1^{i+a} \ : \ i \geq 0 \text{ and } i+a \geq 0 \rbrace .$$
Let $n$ be the pumping lemma constant and set $N := n+|a|$.  
Choose $z = 1^N 0 1^{N+a}$.  Then $|z| \geq n$.  Suppose
$z = uvw$ with $|uv| \leq n$ and $|v| \geq 1$.   Then
$uw = 1^{N-|v|} 0 1^{N+a} \not\in L_a$, since $|v| \geq 1$.  This
contradiction proves that $L_a$ is not regular.
\end{proof}

Interpreted musically, one might say that the positions of every individual
note in the infinity series are determined by a relatively simple program
(specified by a pushdown automaton), but {\it not\/} by the very simplest kind
of program.  There are regularities in these
positions, but not finite-state regularities.

\section{Counting occurrences of notes}

Let us define $r_a (N) = | \{ n : 0 \leq i < 2^N \text{ and } s(i) = a \} |$,
the number of occurrences of the note $a$ in the first $2^N$ positions
of the infinity series.    

\begin{theorem}
For all integers $a$ and all $N \geq 1$ we have
$$r_a (N) =  {{N-1} \choose {\lfloor (N-a)/2 \rfloor}} .$$
\end{theorem}

\begin{proof}
By induction on $N$.  The base case is $N = 1$, whence $N-1 = 0$.
Then ${0 \choose {\lfloor (N-a)/2 \rfloor}}$ is $1$ if
$\lfloor (1-a)/2 \rfloor = 0$ and $0$ otherwise.  This is
$1$ if $a \in \lbrace 0, 1 \rbrace$, and $0$ otherwise.  But the
first $2$ notes of ${\bf s}$ are $0$ and $1$, so the result holds.

Now assume the claim holds for $N' < N$; we prove it for $N$.
Now a value of $a$ in ${\bf s}[0..2^N-1]$ can occur in either an even or
odd position.  If it occurs in an even position, then it arises
from $-a$ occurring in ${\bf s}[0..2^{N-1} - 1]$.  If it occurs in
an odd position, then it arises from $a-1$ occurring in ${\bf s}[0..2^{N-1} - 1]$.  It follows that, for $N \geq 2$, that
$$r_a(N) = r_{-a} (N-1) + r_{a-1} (N-1).$$
Hence, using induction and the classical binomial coefficient identities
$$ {M \choose i} = {{M-1} \choose i} + {{M-1} \choose i-1}, $$
and 
$${M \choose i} = {M \choose {M-i}}  ,$$
we have
\begin{eqnarray*}
r_a (N) &=& r_{-a} (N-1) + r_{a-1} (N-1) \\
&=& {{N-2} \choose {\lfloor (N-1+a)/2 \rfloor} } + 
{{N-2} \choose {\lfloor (N-a)/2 \rfloor} }  \\
&=& 
\begin{cases}
{{N-2} \choose {(N+a)/2 - 1}} + {{N-2} \choose {(N-a)/2}}, &
	\text{if $N \equiv \modd{a} {2}$};  \\
{{N-2} \choose {(N+a-1)/2}} + {{N-2} \choose {(N-a-1)/2}}, &
	\text{if $N \not\equiv \modd{a} {2}$}; 
\end{cases}  \\
&=&
\begin{cases}
{{N-2} \choose {(N-2) - (N+a)/2 + 1}} + {{N-2} \choose {(N-a)/2}}, &
	\text{if $N \equiv \modd{a} {2}$};  \\
{{N-2} \choose {(N-2) - (N+a-1)/2}} + {{N-2} \choose {(N-a-1)/2}}, &
	\text{if $N \not\equiv \modd{a} {2}$}; 
\end{cases}  \\
&=&
\begin{cases}
{{N-2} \choose {(N-a)/2 - 1}} + {{N-2} \choose {(N-a)/2}}, &
	\text{if $N \equiv \modd{a} {2}$};  \\
{{N-2} \choose {(N-a-1)/2 - 1}} + {{N-2} \choose {(N-a-1)/2}},  &
	\text{if $N \not\equiv \modd{a} {2}$}; 
\end{cases}  \\
&=&
\begin{cases}
{{N-1} \choose {(N-a)/2}}, &
	\text{if $N \equiv \modd{a} {2}$};  \\
{{N-1} \choose {(N-a-1)/2}}, & 
	\text{if $N \not\equiv \modd{a} {2}$}; 
\end{cases}  \\
&=&
{{N-1} \choose {\lfloor (N-a)/2 \rfloor}},
\end{eqnarray*}
as desired.
\end{proof}

\begin{corollary}
The limiting density of occurrence of each note is $0$.
\end{corollary}

\begin{proof}
Follows immediately from Stirling's formula.
\end{proof}

Interpreted musically, we conclude that
each individual note occurs more and more sparsely,
as $n \rightarrow \infty$; the sequence $\bf s$ is not uniformly
recurrent.

\section{Note pairs in the infinity series}

We now discuss those pairs $(i,j)$ that occur as two consecutive
notes in the infinity series; we call this a {\it note-pair}.
If there exists $n$ such that
$s(n) = i$ and $s(n+1) = j$, we say that the note-pair
$(i,j)$ is {\it attainable};
otherwise we say it is {\it unattainable}.

\begin{theorem}
The pair $(i,j)$ is attainable if any one of the following
conditions hold:
\begin{itemize}
\item[(a)]  $i > 0$ and $-i \leq j \leq i-1$; 

\item[(b)]  $j \geq 1$ and $1-j \leq i \leq j+1$ and
$i \not\equiv \modd{j} {2}$;

\item[(c)]  $j \leq -2$ and $j+2 \leq i \leq -j-2$ and $i \equiv
\modd{j} {2}$.
\end{itemize}
Otherwise $(i,j)$ is unattainable.
\label{c-thm}
\end{theorem}

\begin{proof}
The proof has two parts.  In the first part we show that if 
integers $i,j$
obey any of the conditions (a)--(c) above, then the pair $(i,j)$
is attainable.
In the second part of the proof, we show that
the remaining pairs are unattainable.

The following cases cover all three cases (a)--(c):

\bigskip

\noindent Case 1:  $j \leq \min(-1,i-1)$ and $i \equiv \modd{j} {2}$.
Take $a = -(j+1)$ and $c = i-(j+1)$.  Then the inequalities
imply $a, c \geq 0$ and the congruence implies that $c$ is odd.
Take $(n)_2 = 1^a 0 1^c$.  Then
$(n+1)_2 = 1^{a+1} 0^c$.  So $s(n) = c-a = i$ and $s(n+1) = -(a+1) = j$.

\bigskip

\noindent Case 2:  $j \geq \max(2, 1-i)$ and $i \not\equiv \modd{j} {2}$.
Take $a = j-1$ and $c = i+j-1$.  Then the inequalities
imply $a, c \geq 0$ and the congruence implies that
$c$ is even.  Take $(n)_2 = 1^a 0 1^c$.  Then
$(n+1)_2 = 1^{a+1} 0^c$.  So $s(n) = c-a = i$ and $s(n+1) = a+1 = j$.

\bigskip

\noindent Case 3:
$i \geq 0$ and $1-i \leq j \leq 1$ and $i \not\equiv \modd{j} {2}$.
Take $a = 1-j$ and $c = i+j-1$.  Then the inequalities
imply $a, c \geq 0$ and the congruence implies that
$c$ is even.  Take $(n)_2 = 1^a 0 0 1^c$.  Then
$(n+1)_2 = 1^a 0 1 0^c$.  So $s(n) = a+c = i$ and $s(n+1) = 1-a = j$.

\bigskip

\noindent Case 4:
$i \geq 0$ and $0 \leq j \leq i-1$ and $i \equiv \modd{j} {2}$.
Take $a = j+1$ and $c = i-(j+1)$.  Then the inequalities
imply $a, c \geq 0$ and the congruence implies that
$c$ is odd.  Take $(n)_2 = 1^a 0 0 1^c$.  Then
$(n+1)_2 = 1^a 0 1 0^c$.  So $s(n) = a+c = i$ and $s(n+1) = a-1 = j$.

\bigskip

Cases (1)--(4) correspond in a somewhat complicated way to
parts (a)--(c) of the theorem above.  Table 1 below illustrates this
correspondence.

\begin{table}[H]
\begin{center}
\begin{tabular}{c|ccccccccccccccccccccccccc}
$i\backslash j$&$\overline{12}$&$\overline{11}$&$\overline{10}$
&$\overline{9}$&
$\overline{8}$&$\overline{7}$&$\overline{6}$&$\overline{5}$
&$\overline{4}$&$\overline{3}$
&$\overline{2}$&$\overline{1}$&0&1&2&3&4&5&6&7&8&9&10&11 \\
\hline
& \\
$\overline{10}$&1&A&5&5&5&5&5&5&5&5&5&5&5&5&5&5&5&5&5&5&5&5&A&2 \\
$\overline{9}$&6&1&A&5&5&5&5&5&5&5&5&5&5&5&5&5&5&5&5&5&5&A&2&6 \\
$\overline{8}$&1&6&1&A&5&5&5&5&5&5&5&5&5&5&5&5&5&5&5&5&A&2&6&2 \\
$\overline{7}$&6&1&6&1&A&5&5&5&5&5&5&5&5&5&5&5&5&5&5&A&2&6&2&6 \\
$\overline{6}$&1&6&1&6&1&A&5&5&5&5&5&5&5&5&5&5&5&5&A&2&6&2&6&2 \\
$\overline{5}$&6&1&6&1&6&1&A&5&5&5&5&5&5&5&5&5&5&A&2&6&2&6&2&6 \\
$\overline{4}$&1&6&1&6&1&6&1&A&5&5&5&5&5&5&5&5&A&2&6&2&6&2&6&2 \\
$\overline{3}$&6&1&6&1&6&1&6&1&A&5&5&5&5&5&5&A&2&6&2&6&2&6&2&6 \\
$\overline{2}$&1&6&1&6&1&6&1&6&1&A&5&5&5&5&A&2&6&2&6&2&6&2&6&2 \\
$\overline{1}$&6&1&6&1&6&1&6&1&6&1&A&5&5&A&2&6&2&6&2&6&2&6&2&6 \\
0& 1&6&1&6&1&6&1&6&1&6&1&A&A&3&6&2&6&2&6&2&6&2&6&2 \\
1& 7&1&7&1&7&1&7&1&7&1&7&1&3&7&2&7&2&7&2&7&2&7&2&7 \\
2& 1&7&1&7&1&7&1&7&1&7&1&3&4&3&7&2&7&2&7&2&7&2&7&2 \\
3& 7&1&7&1&7&1&7&1&7&1&3&1&3&4&2&7&2&7&2&7&2&7&2&7 \\
4& 1&7&1&7&1&7&1&7&1&3&1&3&4&3&4&2&7&2&7&2&7&2&7&2 \\
5& 7&1&7&1&7&1&7&1&3&1&3&1&3&4&2&4&2&7&2&7&2&7&2&7 \\
6& 1&7&1&7&1&7&1&3&1&3&1&3&4&3&4&2&4&2&7&2&7&2&7&2 \\
7& 7&1&7&1&7&1&3&1&3&1&3&1&3&4&2&4&2&4&2&7&2&7&2&7 \\
8& 1&7&1&7&1&3&1&3&1&3&1&3&4&3&4&2&4&2&4&2&7&2&7&2 \\
9& 7&1&7&1&3&1&3&1&3&1&3&1&3&4&2&4&2&4&2&4&2&7&2&7 \\
10& 1&7&1&3&1&3&1&3&1&3&1&3&4&3&4&2&4&2&4&2&4&2&7&2 \\
11& 7&1&3&1&3&1&3&1&3&1&3&1&3&4&2&4&2&4&2&4&2&4&2&7
\end{tabular}
\end{center}
\caption{Illustration of the cases in the proof}
\end{table}

The letter A represents the fact that both cases (5) and (6)
below hold.

The pairs
not covered by conditions (a)-(c) above can
be divided into three parts:

\bigskip

\noindent Case 5:  $i \leq 0$ and $i-1 \leq j \leq -i$.

\bigskip

\noindent Case 6:  $i \leq 0$ and 
($j \geq -i$ and $i \equiv \modd{j} {2}$) or
($j \leq i$ and $i \not\equiv \modd{j} {2}$).

\bigskip

\noindent Case 7:  $i > 0$ and 
($j \leq -(i+1)$ and $i \not\equiv \modd{j} {2}$) or
($j \geq i$ and $i \equiv \modd{j} {2}$).

We need to show all of these pairs are unattainable.

First, we need a lemma:

\begin{lemma}
Suppose $n = 4k+a$ for $0 \leq a \leq 3$, and
$s(n) = i$ and $s(n+1) = j$.
Then the values of $s(k)$, $s(2k)$, $s(2k+1)$, and $s(2k+2)$ are as follows:
\begin{table}[H]
\begin{center}
\begin{tabular}{cccccc}
$a$ & $s(k)$ & $s(2k)$ & $s(2k+1)$ & $s(2k+2)$ & $j$ \\
\hline
$0$ & $i$   & $-i$  & $i+1$ & --- & $1-i$ \\
$1$ & $1-i$ & $i-1$ & $2-i$ & --- & $i-2$ \\
$2$ & $-i-1$ & $i+1$ & $-i$ & --- & $1-i$ \\
$3$ & $i-2$ & $2-i$ & $i-1$ & $-j$ & ---  \\
\end{tabular}
\end{center}
\end{table}
\end{lemma}

\begin{proof}
Follows immediately from the defining recursion.
\end{proof}

We now show that Case 5 cannot occur.  Choose the smallest possible $n$
such that $s(n) = i$ and $s(n+1) = j$, over all
$i, j$ satisfying the conditions
$i \leq 0$ and $i-1 \leq j \leq -i$.  

From Table 2 above we see that if $n = 4k$ or $n = 4k+2$
we have $j = 1-i > -i$,
a contradiction.  Similarly, if $n = 4k+1$ then $j = i-2 < i-1$,
a contradiction.  Hence $n = 4k+3$.  

Now consider $n' := 2k+1 = (n-1)/2 < n$.  Let $i' := i-1$ and
$j' = -j$.  Note that $i' < i \leq 0$ and
$i' - 1 = i - 2 < -j - 2 < -j = j'$.
However $s(n') = i-1 = i'$ and $s(n'+1) = s(2k+2) = -s(4k+4) = -j = j'$,
contradicting the minimality of $n$.

\bigskip

Next we show that Case 6 cannot occur.  Choose the smallest possible $n$
such that $s(n) = i$ and $s(n+1) = j$,
over all $i, j$ satisfying
($j \geq -i$ and $i \equiv \modd{j} {2}$) or
($j \leq i$ and $i \not\equiv \modd{j} {2}$).

From Table 2 above we see that if $n = 4k$ or
$n = 4k+2$ then $j = 1-i > 0$ (since
$i \leq 0$).    So $i \equiv \modd{j} {2}$.  This contradicts $j = 1-i$.
Similarly, if $n = 4k+1$, then $j = i-2$.  Since $i \leq 0$ we have
$j < 0$ and hence $i \not\equiv \modd{j} {2}$.  This contradicts $j = i-2$.
Hence $n = 4k+3$.

Now consider $n' := 2k+1 = (n-1)/2 < n$.  Let $i' := i-1$ and
$j' = -j$.  Note that $i' < 0$.  There are now two subcases to
consider: (i) $j \geq -i$ and $i \equiv \modd{j} {2}$
and (ii) $j \leq i$ and $i \not\equiv \modd{j} {2}$.

\medskip

Subcase (i): The case $j = -i$ is already ruled out by Case 5.
So $j \geq 1-i$.  Then $j' = -j \leq i-1 = i'$.  
Furthermore $j' \not\equiv \modd{i'} {2}$.  
However $s(n') = i-1 = i'$ and
$s(n'+1) = s(2k+2) = -s(4k+4) = -j = j'$, contradicting the minimality
of $n$.

\medskip

Subcase (ii):  The pair where $j = i$ is already unattainable by Case 5.
So $j < i$.  Then $j' = -j > -i$, implying $j' \geq 1-i = -i'$.
Furthermore $i \equiv \modd{j} {2}$.  
Again $s(n') = i-1 = i'$ 
$s(n'+1) = s(2k+2) = -s(4k+4) = -j = j'$, contradicting the minimality
of $n$.

\bigskip

Finally, we now show that Case 7 cannot occur.
Suppose there is a pair of values $(s(n), s(n+1)) = (i,j)$ satisfying
the conditions $i > 0$ and either $j \leq -(i+1)$ and
$i \not\equiv \modd{j} {2}$, or $j \geq i$ and $i \equiv \modd{j} {2}$.
Among all such $(i,j)$, let $J$ be the minimum of the absolute values
of $j$.
Among all pairs of the form
$(i,\pm J)$, let $(I, J')$ be a pair with the smallest value of the
first coordinate (which is necessarily positive).

Suppose $n$ is such that $s(n) = I$ and $s(n+1) = J'$.
If $n = 4k$ or $n = 4k+2$ then from Table 2,
we get $J' = 1-I > (-1)-I$ and $J' = 1-I < 0$,
a contradiction.  If $n = 4k+1$ then from Table 2,
we get $J' = I-2 > -(I+1)$ (since $I > 0$)
and $J' = I-2 < I$, a contradiction.  Hence $n = 4k+3$.  

Then Table 2 implies that if we take $n' = 2k+1 = (n-1)/2 < n$, and
$I' = I-1$ 
and $s(n') = I'$ and $s(n'+1) = -J'$.
If $I'> 0$ then $(I', -J')$ is a pair whose second coordinate has
the same absolute value as $(I,J)$, but whose first coordinate is
smaller, a contradiction.
Otherwise $I' = 0$.  But then the pair
$(0, -J')$ is not attainable by (b), a contradiction.
\end{proof}

Next we consider the possible intervals that can occur in the infinity
series.  

\begin{corollary}
There exists $n$ such that $s(n+1) - s(n) = k$ if and only if
either $k < 0$, or $k > 0$ and $k$ odd.
\label{int}
\end{corollary}

\begin{proof}
Follows immediately from 
Theorem~\ref{c-thm}.
\end{proof}

\begin{corollary}
A note-pair $(a,b)$ never occurs at both an odd and even
position in $\bf s$.
\label{cor2}
\end{corollary}

\begin{proof}
Suppose $(a,b)$ occurs at an even position.  Then from the 
recurrence we have $b = 1-a$.  If it occurs at an 
odd position too, say $n = 2k+1$, then $s(k) = a-1$ and
$s(k+1) = -b$.  Then at position $k$ we have the pair
$(a-1, a-1)$, which by Theorem~\ref{c-thm} does not occur.
\end{proof}

\begin{corollary}
There are never five or more consecutive non-negative notes in $\bf s$.
\end{corollary}

\begin{proof}
Assume there are.  Then there are four consecutive non-negative notes
starting at an even position.  From the recursion,
these four notes, starting at position $n$, are of the form 
$a, 1-a, b, 1-b$.  Now $a \geq 0$ and $1-a \geq 0$ imply $a \in
\lbrace 0, 1 \rbrace$, and similarly for $b$.  
From the recursion, starting at position $n/2$ we must 
have the notes $(-a, -b)$.  However, from Theorem~\ref{c-thm}, none of
the note-pairs $\lbrace (0,0), (-1,0), (0,-1), (-1,-1) \rbrace$
occur.  This contradiction proves the result.
\end{proof}

In a similar fashion, we can prove there are never more than two
consecutive positive notes, or two consecutive negative notes, or
two consecutive non-positive notes, in the infinity series.

\section{Repetitions in the infinity series}

Repetitions in sequences have been an object of intense study
since the pioneering results of Axel Thue more than a hundred years ago
\cite{Thue:1906,Thue:1912,Berstel:1995}.    Thue proved that
the Thue-Morse sequence (i.e., the infinity sequence taken modulo
$2$) is overlap-free:  it contains no block of the form $xxa$ where
$x$ is a nonempty block and $a$ is the first number in $x$.

In this section we characterize close repetitions in the infinity
series.  We prove that the infinity series has an even stronger
avoidance property than the Thue-Morse sequence.

\begin{theorem}
If the infinity series contains a block of the form
$xyx$, with $x$ nonempty, then $|y| \geq 2|x|$.  In particular
$\bf s$ contains no two consecutive identical blocks.
\label{s-thm}
\end{theorem}

\begin{proof}
We call a block of notes of the form $xyx$ with $|y| < 2 |x|$ a {\it proximal
repetition}. Assume, contrary to what we want to prove, that
$\bf s$ has a proximal repetition $xyx$ occurring
for the first time at some position $n$.
Then, without loss of generality,
we can assume that $|y|$ is minimal over all proximal repetitions
occurring in $\bf s$.  Furthermore, we can assume that $n$ is as small
as possible over all occurrences of this $xyx$ in $\bf s$.  Finally,
we can assume that $|x|$ is as small as possible over all $xyx$ 
with $|y|$ minimal occurring at position $n$.
There are a number of
cases to consider.

\bigskip

\noindent Case 1:  $|x| = 1$.  If $\bf s$ contains $xyx$ with $|y| < 2|x|$
then $x = a$ and $y = b$ for single numbers $a, b$.

If $aba$ occurs beginning at an even position $n = 2k$, then from
Observation~\ref{obs1}, we know that $b$ immediately follows the second
$a$.  So $abab$ occurs at position $2k$.  Then from the recurrence
we know that $(-a) (-a)$ occurs at position $k$.  But
from Corollary~\ref{int} we know that this is impossible.

If $aba$ occurs beginning at an odd position $n = 2k+1$, then
from Observation~\ref{obs1}, we know that $b$ immediately precedes the first
$a$ in $\bf s$.  So $baba$ occurs at position $2k$, and we have already
ruled this out in the previous paragraph.

\bigskip

\noindent Case 2:  $|x| \geq 2$ and $|x| \not\equiv \modd{|y|} {2}$.
Then by considering the first
two notes of $x$, say $ab$, we have that $ab$ occurs beginning at both
an odd and an even position, contradicting Corollary~\ref{cor2}.

\bigskip

\noindent Case 3:  $|x| \geq 2$ and both $|x|, |y|$ even.  If the block $xyx$ occurs
starting at an even position $n = 2k$, then $x'y'x'$ occurs
at position $k$. Now $|x'| = |x|/2$ and $|y'| = |y|/2$ and
$x = g(x')$, $y = g(y')$, so $x'y'x'$ is a proximal repetition occurring
at position $k$.
If $|y| > 0$, then $|y'| < |y|$, contradicting our assumption that
$|y|$ was minimal.  If $|y| = 0$, and $n > 0$, then $x'y'x'$ occurs 
at position $n/2 < n$, contradicting the assumption that our $xyx$ occurs
at the earliest possible position.  Finally, if $|y| = 0$ and $n = 0$, then
$|x'| < |x|/2$, contradicting the minimality of $|x|$.

Otherwise $xyx$ occurs starting at an odd position $n = 2k+1$.
If $|y| = 0$ then write $x = wa$ for a single number $a$.  Since
$|x|$ is even and $xyx = wawa$ occurs beginning at an odd position,
the first $a$ is at an even position.  So another $a$ immediately
precedes the first $w$.  Then $awaw$ occurs starting at position $n = 2k$.
This contradicts our assumption that $xyx$ was the earliest occurrence.

Otherwise $|y|> 0$.  Write $x = a w$ for a single letter $a$ and
$y = z b$ for a single letter $b$.  Note that $|z|$ is odd.
Since both $|x|$ and $|y|$
are even, $b$ occurs at an even position and immediately precedes
the second occurrence of $x$.  So $b$ also immediately precedes
the first occurrence of $x$.  Thus $bxyx = bawzbaw$ occurs
at position $2k$.  Then
$|z| < |y| < 2|x| < 2|baw|$, so $(baw) z (baw)$ is a proximal repetition
with $|z| < |y|$. This contradicts our assumption that
$|y|$ was minimal.

\bigskip

\noindent Case 4:  $|x| \geq 2$ and both $|x|, |y|$ odd.  
Suppose $xyx$ begins at an even position, say $n = 2k$.  Then, writing
$y = az$ for a single number $a$,
we see that $a$ immediately follows the first $x$
and occurs at an odd position.  So $a$ also follows the second $x$
and we know $xazxa$ occurs at position $n$.  Since $|xa|$ and $|z|$ are
both even, there exist $x', y'$ with $g(x') = xa$ and $g(y') = z$.
So $x' y' x'$ occurs at position $k$.  However
$|y'| = |z|/2 = (|y|-1)/2 < |x| - 1/2 < |x| = 2|x'| - 1$, and so
$x'y'x'$ is a proximal repetition with $|y'| < |y|$.  This contradicts
our assumption that $|y|$ was minimal.

Similarly, if $xyx$ begins at an odd position, say $n= 2k+1$, then
we can write $y = za$ for a single number $a$.  Then $a$ occurs at
an even position and immediately precedes the second $x$, so it also occurs
immediately before the first $x$.  Thus $axzax$ occurs at position $n = 2k$.
Then we can argue about $ax$ and $z$ exactly as in the preceding paragraph
to get a contradiction.
\end{proof}

We remark that Theorem~\ref{s-thm} is optimal since, for example,
at position $1$ of $\bf s$ we have $(1, -1, 2, 1)$, which corresponds
to $x = 1$, $y = (-1, 2)$ and $|y| = 2 |x|$.  By applying $g$ to this
occurrence we find larger and larger occurrences of $xyx$ satisfying the
same equality.

Musically, we may say that although each phrase in the infinity series
occurs infinitely often, we never hear exactly the same phrase twice without
a relatively long delay between the two occurrences.  This may partially
account for the impression of neverending novelty in the music.

\section{Characterizing the infinity series}

The infinity series has been called unique \cite{Mortensen:2014c}.
To make this kind of assertion mathematically rigorous, however, we
need to decide on the main properties of the sequence to see if there
could be other sequences meeting the criteria.

Although these main properties could be subject to debate, here
are a few of the properties of $\bf s$ observed by us and others:

\bigskip

1.  $s$ is a surjective map from $\Enn$ to $\Zee$.
That is, for all $a \in \Zee$ there
exists $n$ such that $s(n) = a$.

\bigskip

\begin{proof}
We have $s(0) = 0$.  If $a > 0$, then it is easy to see 
that $s(2^a - 1) = a$.
If $a < 0$, then it is easy to see that $s(1-2^{-a}) = a$.
\end{proof}

\bigskip

2.  It is $k$-self-similar \cite{Mortensen:2014c}.  That is, there exists a $k \geq 2$
such that for all $i \geq 0$ and
$0 \leq j < k^i$, the subsequence $(s(k^i n + j))_{n \geq 0}$
is either of the form $s(n) + a$ for some $a$ (in other words, the
sequence transposed by $a$ half-steps) or of the form
$-s(n) + a$ for some $a$ (in other words, the sequence inverted
and then transposed by $a$ half-steps).  In the \ng\ sequence, we
have $k = 2$.

\bigskip

3.  Every interval occurs.  That is, for all $i \not= 0$,
there are two consecutive notes that are exactly $i$ half-steps apart.
More precisely, for all $i > 0$ there exists $n$ such that
$$|s(n+1)-s(n)| = i.$$
Note, however, that by Theorem~\ref{c-thm}, it is {\it not\/} true that every
possible note-pair occurs somewhere.  In fact, asymptotically, only
half of all possible note-pairs occur.  Furthermore, some intervals
occur only in a descending form; by Corollary~\ref{int} this is true exactly
of all even intervals.

\bigskip

4.  Every interval that occurs, occurs beginning at infinitely many
different notes.  That is, for all negative $i$ and all positive
odd $i$, there are infinitely many distinct $j$ such that there 
exists $n$ with
$s(n) = j$ and $s(n+1) = i+j$.

\bigskip

5.  Runs of consecutive negative (resp., positive, non-negative, non-positive)
notes are of bounded length.

\bigskip

6.  It is recurrent; that is, to say, every block of values that
occurs, occurs infinitely often.

\begin{proof}
Let $(a_0, a_1, \ldots, a_{j-1})$ be a block of $j$ consecutive
values of the infinity series, for some $j \geq 1$, that is,
suppose there exists $n$ such that $s(n+i) = a_i$ for $0 \leq i < j$.
Then there exists a power of $2$, say $2^N$, such that
$j \leq 2^N$.   It therefore suffices to show that
the block $B := (s(0), s(1), \ldots, s(2^N-1))$ appears infinitely
often.

Consider the block 
	$$ A_t = (s(5 \cdot 2^{N+t}), s(5 \cdot 2^{N+t} + 1), \ldots,
	s(5 \cdot 2^{N+t} + 2^N - 1)) .$$
We claim that $A_t = B$ for all $t \geq 0$.
To see this, note that the binary expansion of
$5 \cdot 2^{N+t} + i$, for $0 \leq i < 2^N$, looks like
$101 0^t$ followed by $w$, where $w$ is the binary expansion of
$i$ padded on the left with zeros to make its length $N$.
It now follows from Lemma~\ref{lem1} that $s(5 \cdot 2^{N+t} + i) =  
a_i$ for $0 \leq i < 2^N$, thus producing a new occurrence of 
$B$ for each $t \geq 0$.
\end{proof}

\bigskip

7.  It is slowly-growing, that is, $|s(n)| = O(\log n)$, and
furthermore there exists a constant $c$ such
that $|s(n)| > c \log n$ infinitely often.  Musically, this corresponds
to novel notes appearing infinitely often, but with longer and longer
delays between their first appearances.

\bigskip

8.  It is non-repetitive or ``squarefree''.  That is, for all
$n \geq 1$ and $i \geq 0$ the phrase given by the notes
$(s(i), s(i+1), \ldots, s(i+n-1))$
is never followed immediately by the same phrase repeated again.
We proved an even stronger statement in Theorem~\ref{s-thm}.

\bigskip

Are there other sequences with the same eight properties?  A brief computer
search turned up many others.  For example, consider the sequence
${\bf t} = (t(n))_{n \geq 0}$
given by the rules
\begin{eqnarray*}
t(0) &=& 0 \\
t(4n) &=& t(n) \\
t(4n+1) &=& t(n) - 2 \\
t(4n+2) &=& -t(n) - 1 \\
t(4n+3) &=& t(n) + 2 .
\end{eqnarray*}

The first few terms are
\begin{table}[H]
\begin{center}
\begin{tabular}{c|ccccccccccccccccccccc}
$n$ & 0 & 1 & 2 & 3 & 4 & 5 & 6 & 7 & 8 & 9 & 10 & 11 & 12 & 13 & 14 &
15 & 16 & 17 & 18 \\
\hline 
$t(n)$ & 0 & $-2$ & $-1$ & 2 & $-2$ & $-4$ & 1 & 0 & $-1$ & $-3$ & 0
& 1 & 2 & 0  &$-3$  &4  &$-2$  &$-4$  &1  \\
\end{tabular}
\end{center}
\end{table}

It is not hard to prove, along the lines of the
the proofs given above, that $\bf t$ shares all eight properties with $\bf s$.
So in fact, $\bf s$ is not really unique at all.

\section{Variations}

\ng\ also explored variations on the 
infinity series.
The descriptions given in \cite{Mortensen:2014d,Mortensen:2014e} are 
a little imprecise, so we reformulate them below.

First variation:
\begin{eqnarray*}
u(0) &=& 0 \\
u(3n) &=& -u(n) \\
u(3n+1) &=& u(n) - 2\\
u(3n+2) &=& u(n) - 1 
\end{eqnarray*}

The first few terms are
\begin{table}[H]
\begin{center}
\begin{tabular}{c|ccccccccccccccccccccc}
$n$ & 0 & 1 & 2 & 3 & 4 & 5 & 6 & 7 & 8 & 9 & 10 & 11 & 12 & 13 & 14 &
15 & 16 & 17 & 18 \\
\hline 
$u(n)$ & 0 & $-2$ & $-1$ & 2 & $-4$ & $-3$ & 1 & $-3$ & $-2$ & $-2$
& 0 & 1 & 4 &
$-6$ & $-5$ & 3 & $-5$ & $-4$ & $-1$ \\
\end{tabular}
\end{center}
\end{table}

Note that this sequence has a repetition of order three (``cube'')
beginning at position $32$:  $(-1,-1,-1)$.  This, together with the
recursion, ensures that there will be arbitrarily large such
repetitions in the sequence.  It fails property 8 of the previous section.

Second variation:
\begin{eqnarray*}
v(0) &=& 0 \\
v(3n) &=& -v(n) \\
v(3n+1) &=& v(n) - 3 \\
v(3n+2) &=& -2-v(n) 
\end{eqnarray*}

\begin{table}[H]
\begin{center}
\begin{tabular}{c|ccccccccccccccccccccc}
$n$ & 0 & 1 & 2 & 3 & 4 & 5 & 6 & 7 & 8 & 9 & 10 & 11 & 12 & 13 & 14 &
15 & 16 & 17 & 18 \\
\hline 
$v(n)$ & 0 & $-3$ & $-2$ & 3 & $-6$ & 1 & 2 & $-5$ & 0 & $-3$ & 0 & $-5$ & 6 &
$-9$ & 4 & $-1$ & $-2$ & $-3$ & $-2$ \\
\end{tabular}
\end{center}
\end{table}

This sequence fails to have property 3:  not every interval occurs.  In
fact, only intervals of an odd number of half steps occur.  However, it
has all the other properties listed in the previous section.

\end{document}